\documentclass[11pt]{amsart}
\usepackage{amsmath,amssymb, amscd}

\usepackage{graphicx}

\usepackage{verbatim}

\usepackage[all]{xy} 
\usepackage{todonotes}      

\def\sqr#1#2{{\vcenter{\hrule height.#2pt
        \hbox{\vrule width.#2pt height#1pt \kern#1pt
                \vrule width.#2pt}
        \hrule height.#2pt}}}

\newtheorem{theorem}{Theorem}[section]
\newtheorem{lemma}[theorem]{Lemma}
\newtheorem{proposition}[theorem]{Proposition}
\newtheorem{corollary}[theorem]{Corollary}

\theoremstyle{definition}
\newtheorem{definition}[theorem]{Definition}
\newtheorem{example}[theorem]{Example}

\newtheorem{question}[theorem]{Question}

\newtheorem{remark}[theorem]{Remark}

\newtheorem{setting}[theorem]{Setting}

\newtheorem{discussion}[theorem]{Discussion}

\newtheorem{defrmk}[theorem]{Definition-Remark}

\DeclareMathOperator{\m}{\mathfrak m}
\DeclareMathOperator{\p}{\mathfrak p}
\DeclareMathOperator{\q}{\mathfrak q}

\DeclareMathOperator{\Spec}{Spec}

\DeclareMathOperator\ord{ord}

\DeclareMathOperator\Min{Min}

\DeclareMathOperator\Proj{Proj}
\DeclareMathOperator{\hgt}{ht}

\DeclareMathOperator{\Rees}{Rees }
\DeclareMathOperator{\Cl}{Cl}

\DeclareMathOperator{\SI}{\mathcal I}
\DeclareMathOperator{\SO}{\mathcal O}

\def\alert#1{\smallskip{\hskip\parindent\vrule%
\vbox{\advance\hsize-2\parindent\hrule\smallskip\parindent.4\parindent%
\narrower\noindent#1\smallskip\hrule}\vrule\hfill}\smallskip}

\begin{document}

\title[Blowing up finitely supported complete ideals ]
{Blowing up finitely supported complete ideals \\ in a regular local ring }


\author{William Heinzer}
\address{Department of Mathematics, Purdue University, West
Lafayette, Indiana 47907 U.S.A.}
\email{heinzer@purdue.edu}

\author{Youngsu Kim}
\address{Department of Mathematics, University of California, Riverside,
Riverside, California 92521 U.S.A}
\email{youngsu.kim@ucr.edu}

\author{Matthew Toeniskoetter}
\address{Department of Mathematics, Purdue University, West
Lafayette, Indiana 47907 U.S.A.}
\email{mtoenisk@purdue.edu}

\date \today


\maketitle
\bigskip

\begin{abstract}  Let $I$ be a finitely supported complete $\m$-primary ideal of a  regular local 
ring  $(R,\m)$.   We consider singularities   of  the projective  models 
$\Proj R[It]$ and  $\Proj \overline{R[It]}$  over $\Spec R$, 
where  $\overline{R[It]}$ denotes the integral closure of the Rees algebra $R[It]$.
 A  theorem of  Lipman  implies that  the ideal  $I$ has a unique factorization as 
a $*$-product of special $*$-simple complete ideals with possibly negative exponents for
some of the factors.    If  $\Proj \overline{R[It]}$  is regular,   we prove that   
$\Proj \overline{R[It]}$ is the 
regular model obtained by blowing up the finite set of base points of $I$.  
Extending work of Lipman and Huneke-Sally in dimension $2$, we 
prove that every local ring  $S$ on  $\Proj \overline{R[It]}$  that is a unique factorization domain 
is regular.   Moreover, if  $\dim S \ge 2$  and $S$  dominates $R$,  then $S$  is an infinitely near point to $R$,
that is,  $S$ is obtained from $R$ by a finite sequence of local quadratic transforms.  
\end{abstract}

\baselineskip 18 pt

\section{Introduction} \label{c1}

Let  $(R,\m)$ be a regular local ring  of dimension at least 2.
 A regular local ring $S$ that 
dominates $R$ is  {\it infinitely near to} $R$ if $\dim S \ge 2$ and $S$  may be obtained from $R$ by a finite sequence (possibly empty) 
of local quadratic transforms. 
 An infinitely near point $S$ to $R$ is a {\it base point} of    an  ideal $I$  of $R$ if the transform $I^{S}$ of $I$ in $S$ is 
a proper ideal of $S$. 
The set of base points of  an ideal $I$  of $R$ is denoted $\mathcal{BP}(I)$,  and the ideal $I$ is said to be {\it finitely supported} if 
the set $\mathcal{BP}(I)$ is finite.

The infinitely near points to $R$ form a partially ordered set with respect to domination. 
The regular local ring $R$ is the unique
minimal point     with respect to this  partial order.  For an ideal $I$ of $R$,  the set 
$\mathcal{BP}(I)$ of base points of $I$ is a partially ordered subset of the   set of infinitely near points to $R$. 
If the set $\mathcal{BP}(I)$ is finite, we refer to the maximal regular local rings in  $\mathcal{BP}(I)$ as 
{\it terminal base points} of $I$.   If $I$ is a finitely supported ideal, then results of Lipman  \cite[Prop.~1.21, Cor.~1.22]{L}
imply that $\dim S = \dim R$ and $S/ I^S$ is Artinian for 
each base point $S$ of $I$. 
In particular, the ideal $I$ is $\m$-primary.

\begin{definition} \label{1.1}
Let $(R, \m)$ be a regular local ring and let  $I$ be a finitely supported $\m$-primary ideal. 
Let $\Gamma := \mathcal{BP}(I)$ denote the finite set of base points of $I$.  
By successively blowing up the maximal ideals of the   points in  $\Gamma $  we obtain a  regular projective  model\footnote{ We are using the language of 
Section~17,  Chapter VI of 
Zariski-Samuel \cite{ZS2}.   Thus, for $R$ a subring of a field $K$ and  $A$ a finitely generated 
$R$-subalgebra of
 $K$, the   {\it affine model} over $R$  associated to $A$  is the set of local rings 
$A_P$, where  $P$ varies over the set of prime ideals of $A$. A  {\it model}  $M$   over $R$ is a  subset of  
the local subrings of $K$ that contain $R$ that has   the properties:
(i)  $M$ is 
a finite union of affine models over $R$,  and (ii) each valuation ring of $K$ that contains $R$ dominates at most one of the local rings in $M$.  This second condition is called {\it irredundance}.   A model $M$ over $R$ is said to 
be {\it complete} if each valuation ring of $K$ that contains $R$ dominates a local ring in $M$.    The model $M$ is said to be {\it projective} over $R$ if  there exists a finite
set $a_0, a_1, \ldots, a_n$ of nonzero elements of $K$ such that $M$ is the union of the affine models defined  by the rings 
$A_i  = R[\frac{a_0}{a_i}, \frac{a_1}{a_i}, \ldots, \frac{a_n}{a_i}],  i = 0, 1, \ldots, n$. 
The models  we consider are either affine or projective models over a Noetherian integral domain $R$. 
In the language of schemes, these  models  correspond, respectively,   to affine or projective schemes over $\Spec R$.
}  
 $X_{\Gamma}$  over $R$  
and a projective morphism $X_\Gamma \to \Spec R$.  We call $X_\Gamma$ the \textit{saturated regular model} associated to 
the ideal $I$, or more precisely, to the  set $\Gamma = \mathcal{BP}(I)$.   
\end{definition}

The model $X_\Gamma$ may be obtained by first blowing up the maximal ideal $\m$ of $R$
to obtain the regular model $\Proj R[\m t]  = X_1$.  Each infinitely near  point $S$ in  $ \Gamma = \mathcal{BP}(I)$,  
other than $R$, dominates a unique point on the model $X_1$.  The points in $\Gamma$ in the first neighborhood of $R$ 
are obtained from $R$ by one local quadratic transform and  are points on the model
$X_1$.  Each infinitely near point $S$ in $\Gamma  \setminus \{R\}$  is either a point on the model $X_1$ or is an infinitely near point to a 
unique point $S_1$,  where $S_1$ is a point on the model $X_1$.  
Associated to each infinitely near point $S_1 \in X_1$ such that $\dim S_1 = \dim R$,
 there exists a unique coherent $\SO_{X_1}$-ideal sheaf $\SI$ such that the stalk $\SI_{S_1}$
  is the maximal ideal of $\SO_{X_1, S_1}$ and the stalk $\SI_{T} = \SO_{X_1, T}$ for each point $T$ in $X_1 \setminus \{ S_1 \}$ \cite[Lemma~2.3]{L}.

On $X_1$, we blow up the ideal sheaf that is the product of the ideal sheaves that correspond to the points
 $S_1 \in \Gamma  \cap X_1$ to 
obtain the regular model $X_2$. There exist associated projective morphisms $X_2 \to X_1 \to \Spec R$. 
We continue this process to obtain the 
regular model $X_{\Gamma}$ and projective morphism $X_{\Gamma} \to \Spec R$  in which each of the infinitely near points in 
$\Gamma $ has been blown up.

Let $S$ be a regular local ring  infinitely near to $R$. If $S \neq R$, then $S$ is infinitely near to a point on $\Proj R[\m t]$. 
More generally,  if  $S$  is  a regular local ring  infinitely near to $R$,  then $S$  is infinitely near to a point on $X_\Gamma$
 if and only if   $S$ is not a base point of $I$,  that is,   if and only if $S \notin \Gamma$.

The model $X_\Gamma$ and projective morphism $X_\Gamma \to \Spec R$ have the property that 
for each finitely supported ideal $J$ of $R$ such that $\mathcal{BP}(J) \subseteq    \Gamma$, 
there exists a morphism $f:  X_\Gamma   \to  \Proj \overline{R[Jt]}$
that gives the following   commutative diagram:

\begin{equation*}  \label{eq1.1} \tag{1.1}
\xymatrix{
  X_\Gamma   \ar[d]^{f}    \ar[dr]&  \\
 \Proj \overline{R[Jt]}  \ar[r] & \Spec R.
}
\end{equation*}

\begin{remark}\label{Gamma}  Let $R$ be a regular local ring with $\dim R \ge 2$.  A finite set 
$\Gamma$  of points infinitely near to $R$ is the set of base points of a finitely supported
ideal  of $R$ if and only if the set $\Gamma$ satisfies the following conditions.
\begin{enumerate}
\item  $R \in \Gamma$,  
\item  For each $S \in \Gamma$, 
we have $\dim S = \dim R$, and 
\item  For each $S \in \Gamma$,  each of the regular local rings in the unique chain of 
local quadratic transforms from $R$ to $S$ is in $\Gamma$.
\end{enumerate}

Fix a finite set $\Gamma$  of infinitely near points to $R$ that satisfies these  $3$ conditions. 
For each $R_i  \in \Gamma$, let $P_i$ denote the special $*$-simple ideal associated to the 
pair $R \prec R_i$ \cite[Prop.~2.1]{L}. 
Setting $I = \prod_{R_i \in \Gamma} P_i$, we have $\mathcal{BP}(I) = \Gamma$ and $\Proj R[It] = X_{\Gamma}$.
 
\end{remark}

\begin{defrmk}\label{blowup}
Let $I$ be a finitely supported ideal of a regular local ring $R$.   
We say the morphism $f: X_\Gamma \to \Proj  \overline{R[It]}$  of Diagram~\ref{1.1} 
 is {\it biregular at} $S \in \Proj  \overline{R[It]}$
if $f^{-1}(S) = \{ S \}$.

 We say that  the ideal  $I$   has a {\it saturated factorization} if   
 $\Proj  \overline{R[It]}$ is the  regular model $X_\Gamma$,  where $\Gamma := \mathcal{BP}(I)$. 
 In the case where $\dim R = 2$, this terminology is equivalent to that used in \cite[Def. 5.11]{HJLS}.
  As observed in Remark \ref{Gamma}, there exist 
 finitely supported ideals $I^*$ such that $\mathcal{BP}(I^*) = \mathcal{BP}(I)$ and  $\Proj R[I^*t] = X_{\Gamma}$.
\end{defrmk}

In Section~\ref{sec2} we obtain in Corollary~\ref{2.14} the following result: 

\begin{theorem}   \label{1.3}
Let $I$ be a finitely supported  $\m$-primary  ideal of a regular local ring $(R,\m)$.  
   If  $\Proj \overline{R[It]}$   
 is regular, then  $I$ has a saturated factorization,  that is 
  the  morphism  $ f:  X_\Gamma ~  \to   ~  \Proj \overline{R[It]} $  of Diagram~\ref{1.1}
 is an isomorphism.  More generally,    if each local ring $S \in \Proj \overline{R[It]}$  is  a 
unique factorization domain, then $\Proj \overline{R[It]}$ is regular and $I$ has a 
saturated factorization.

\end{theorem} 

\begin{remark}  \label{1.31}
 Classical results of Zariski and Abhyankar imply Theorem~\ref{1.3}
in the case where $\dim R = 2$,  cf.  \cite[Prop. 5.12]{HJLS}. 
The assertion about unique factorization domains is also known in
the case of dimension 2   \cite[Prop.~3.1]{L1969}  and \cite[Cor.~1.2]{HS}
\end{remark}

We follow the notation of \cite{L}.
For a Noetherian local domain $(R, \m)$, we denote by $\ord_{R}$ the order function defined by the powers of $\m$.
If $R$ is a regular local ring, or more generally if the associated graded ring of $R$ with respect to $\m$ is a domain, then $\ord_{R}$ defines a rank $1$ discrete valuation on the field of fractions of $R$.
For an ideal $I$ in a Noetherian domain, we denote by $\Rees I$ the Rees valuations of $I$.
For ideals $I$ and $J$ of an integral domain $R$, their $*$-product, denoted $I * J$, is the integral closure $\overline{I J}$ of their ordinary product.
For an ideal $I$ of a UFD $R$ with $\hgt I \ge 2$   and $S$  a UFD overring of $R$ with the same field of fractions, the \emph{transform} of $I$ in $S$, denoted $I^{S}$, is $a^{-1} I S$, where $aS$ is the smallest principal ideal
of $S$  containing $I S$.

\begin{discussion}  \label{1.6}
Let $I$ be a finitely supported complete  $\m$-primary ideal  of a regular local domain 
$(R, \m)$  and let $\Gamma := \mathcal{BP}(I)$.   We have:
\begin{enumerate}
\item
  If $\dim R = 2$,  then associated to each
simple complete ideal factor  $J_i$   of $I$, there exists an infinitely near point 
$R_i \in \Gamma$ such that $J_i   = P_{RR_i}$ is the special $*$-simple ideal   associated to the 
pair $R \prec R_i$ \cite[Prop.~2.1]{L}. 
Furthermore, $\Rees J_i = \{\ord_{R_i}\}$ (cf.  \cite[Prop.~14.4.10]{SH}), so the Rees valuation rings of $I$ are in one-to-one correspondence with the distinct simple factors of $I$ as a product of simple complete ideals 
(cf. \cite[Prop.~10.4.8]{SH}).

\item 
In the case where $\dim R \ge 3$, $\Rees I \subseteq    \{\ord_{R_i}~|~ R_i \in \Gamma \}$ \cite[Prop. 4.3]{HK1}.
A special $*$-simple ideal $P_{RR_i}$ associated to an infinitely near point $R_i$ contains $\ord_{R_i}$, but it often
contains additional Rees valuation rings $\ord_{R_j}$ with $R_j \in \Gamma$ and $R_j \ne R_i$.
\item  
If $I$ has saturated factorization, then $\Rees I = \{ \ord_{R_i} \mid R_i \in \Gamma \}$.
If $\dim R = 2$, then this property characterizes the ideals with a saturated factorization.
However, if $\dim R \ge 3$, then Item~2 allows the construction of finitely supported complete ideals $I$ such that 
$\Rees I = \{ \ord_{R_i} \mid R_i \in \Gamma \}$ and yet $I$ does not have a saturated factorization,  that is,
there exist normal local domains $S \in \Proj \overline{R[It]}$ that are singular,  see Example~\ref{samerees}.
\item
Assume that that $\ord_{R_i} \in \Rees I$ for each $R_i \in \Gamma$.
We prove in  Theorem~\ref{3.10} that the morphism 
$ f:  X_\Gamma ~  \to   ~  \Proj \overline{R[It]} $  of Diagram~\ref{1.1}
is biregular at  each  $S \in \Proj \overline{R [I t]}$    that 
 has torsion divisor class group.
Therefore for each singular point $S \in \Proj \overline{R [I t]}$, the divisor class group of $S$ is nontorison.
Thus the singular local domains in examples such as
 Example~\ref{samerees} must   have nontorision divisor class group.
 \end{enumerate}
 \end{discussion}  

\section{Regular blowup implies   saturated factorization } \label{sec2}

In Theorem~\ref{2.13}, we prove that a local UFD on the blowup of a finitely supported ideal is regular.
Lemma~\ref{2.10} and Corollary~\ref{2.11} are special cases of Lemma~\ref{2.40}, where we prove that a UFD $S$ on the blow-up of on a finitely supported $\m$-primary ideal $I$ has the property that $\m S$ is principal.
It follows that $S$ dominates a unique point on $\Proj R [\m t]$, allowing us to apply an inductive argument on the number of base points of $I$.

\begin{lemma} \label{2.10}
Let   $(R,\m)$ be a regular local domain, let $I$ be an $\m$-primary ideal, and let $S$ be a UFD on $\Proj \overline{R [I t]}$ that dominates $R$.
If there exists a DVR $V$ that dominates $S$ such that $\m V$ is the maximal ideal of $V$, then $\m S$ is principal.
It follows that $S$ dominates a local ring on $\Proj R [\m t]$.
\end{lemma}  

\begin{proof}
Since $I$ is $\m$-primary and $S \in \Proj \overline{R[It]}$, the ideal $\m S$ has height one.  
Since $S$ is a UFD,
$\m S = \rho J$,  where $\rho$ is a nonzero nonunit of $S$ and either $J = S$ or $J$ is an ideal of $S$ 
with $\hgt J \ge 2$.   Since $V$ dominates $S$ and $\m V$ is the maximal ideal of $V$,  we must have
$\m V = \rho V$ and $J = S$. 
\end{proof}

\begin{corollary}  \label{2.11} 
 Let   $(R,\m)$ be a regular local domain  and  let $V = \ord_R$.   Let 
 $I$ be  an $\m$-primary ideal and  let $S$ be  the local domain   on $ \Proj \overline{R[It]}$ 
 dominated by $V$.  If $S$   is a UFD, then $S = V$ and   $V \in \Rees I$. 
\end{corollary}  

\begin{proof}
Lemma~\ref{2.10} implies that $S$ dominates a local ring on $\Proj R[\m t]$. Since 
$V = \ord_R \in \Proj R[\m t]$ and $V$ dominates $S$, we have $S = V$. 
\end{proof}  

Lemma~\ref{lem rees mI} compares the Rees valuations of an ideal $I$ with those of the transforms of $I$ in the first neighborhood of $R$.

\begin{lemma}\label{lem rees mI}
Let $(R, \m)$ be a  regular local domain  with $\dim R \ge 2$ and let $I$ be an $\m$-primary ideal 
that has only finitely many base points   $R_1, \ldots, R_n$
 in the first neighborhood of $R$.
Then we have 
\begin{equation}
\{\ord_{R}\}  ~ \cup  ~ \Rees I  ~ =    ~ \{ \ord_{R}\}  ~ \cup   ~   \bigcup_{i=1}^{n} \Rees I^{R_i},
\end{equation}
	where $I^{R_i}$ denotes the transform of $I$ in $R_i$.  
Furthermore, we have 
\begin{equation}
	\Rees \m I   ~ =  ~ \Rees \m ~\cup ~ \Rees I ~= ~  \{\ord_{R}\}  ~ \cup  ~   \Rees I.
\end{equation}
\end{lemma}
\begin{proof}
Let $R'$ be a base point of $I$ in the first neighborhood of $R$.
If $\dim R' < \dim R$,  then there exist infinitely many regular local rings $S$ in $\Proj R[\m t]$ such that $S \subset R'$. 
Notice that $R' = S_{\m_{R'} \cap S}$, where $\m_{R'}$ is the maximal ideal of $R'$. 
Since taking transforms is transitive \cite[Prop.~1.5.iv]{L} and $R'$ is a base point of $I$, 
the ring $S$ is also a base point of $I$. 
Thus the hypothesis that $I$ has only finitely many base points in the first neighborhood of $R$
implies that $\hgt I^{R_i} = \dim R_i = \dim R$ for each $i \in \{1, \ldots, n\}$.  

The $\supseteq$ inclusion of Equation~1 follows from \cite[Prop.~3.11]{HKT}.

To see the $\subseteq$ inclusion of Equation~1, let $V \in \Rees I$, and let $R'$ denote the local quadratic transform of $\m$ along $V$.
If $I R'$ is not principal, then $R'$ is a base point of $I$, so $R' = R_i$ for some $i$.
Thus $V \in \Rees I^{R_i}$ by \cite[Prop.~3.11]{HKT}.
If $I R'$ is principal, then $R'$ dominates the blowup of $I$ along $V$, that is, $R'$ dominates $V$. Therefore we have $R' = V$ and this implies $V = \ord_{R}$.

For Equation~2,
we have $\Rees \m = \{\ord_R\}$,  and it is true in general that the Rees valuation rings of a product $IJ$
of nonzero ideals $I$ and $J$  includes $\Rees I \cup \Rees J$ (\cite[Prop.~10.4.5]{SH}).  The reverse inclusion holds in Equation~2 by Equation~1, since $\m I$ and $I$ have the same transform in each $R_i$.
\end{proof}

\begin{lemma}\label{2.40}
Let $(R, \m)$ be a regular local domain with $\dim R \ge 2$ and let $I$ be an $\m$-primary ideal of $R$ with finitely many base points in the first neighborhood.
Let $S$ be a local ring on $\Proj \overline{R[It]}$ that dominates $R$.
If either
\begin{enumerate}
\item $S$ is a UFD, or
\item $S$ has torsion divisor class group and either $\ord_{R} \in \Rees I$ or $S \not\subset \ord_{R}$,
\end{enumerate}
then $\m S$ is principal.
\end{lemma}

\begin{proof}
Assume that $S$ is a UFD.    By way of contradiction,   
suppose  that $\m S$ is not principal.  We may write $\m S = \rho J$,
where  $\rho$ is a nonzero nonunit  of $S$ and  $J$ is a proper ideal  of $ S$ with $\hgt J \ge 2$.
Let $V$ be a Rees valuation ring of $J$.  
The ring $S$ is on the blowup of $I$ and the center of $V$ on $S$ is of height at least two since it contains $J$. Therefore we have $V \notin \Rees I$.

Notice that the blowup of $\m I$ may  be obtained by first  blowing up $I$ and then blowing up the  the extension  of $\m$  to  the model  $\Proj \overline{R[It]}$.  
 It follows that $V  \in \Rees \m I$. 
 By  Lemma \ref{lem rees mI},  we must have $V = \ord_R$.  
The ideals $\rho S$ and $IS$ have the same radical and $I$ is contained in the maximal ideal of $V$.
However, this implies that $\m V = \rho J V$ is properly contained in the maximal ideal of $V$. 
This contradicts  the fact that $\m V$ is the maximal ideal of $V = \ord_R$.  
We conclude that $\m S$ is principal.

Assume that $S$ has torsion divisor class group.  As in the previous case, suppose that $\m S$ is not principal.
Since $\hgt \m S = 1$, there exists an integer $n > 0$ such that $\m^n S = \rho J$ for some element $\rho \in S$ and some ideal $J $ of $S$, where either $J = S$ or  $\hgt J \ge 2$.  
Since $S$ is a local domain, invertible ideals  in $S$  are principal.  Moreover, if  a power of an ideal is
invertible,  then the ideal is invertible.  Thus the condition that $\m S$ is not principal implies 
that $\m^n S$ is not principal, so $J$ is a proper ideal of $S$.
By the same argument as the previous case, it follows that $V$ is a Rees valuation of $\m^n I$.
By Lemma~\ref{lem rees mI},  we have  $V = \ord_{R}$.
This implies that $S \subset \ord_{R}$ and that $\ord_{R} \notin \Rees I$,   which completes the proof.
\end{proof}

Let $S$ be a local unique factorization domain on the blowup of a finitely supported ideal of a regular
local domain $R$.   In Theorem~\ref{2.13},  we generalize to the case where $\dim R > 2$ a result  
of    Huneke and Sally \cite[Cor.~1.2]{HS} 
 and Lipman \cite[Prop. 3.1]{L1969} for the case where $\dim R = 2$.  If $\dim R = 2$, then every 
 $2$-dimensional local UFD
 that birationally dominates   $R$ 
 is on the blowup of a finitely supported ideal of $R$.

\begin{theorem}  \label{2.13}
 Let   $(R,\m)$ be a regular local domain with $\dim R \ge 2$ and let 
 $I$ be a finitely supported  $\m$-primary ideal. Let 
$S$  be a local ring on $\Proj \overline{R[It]}$ that dominates $R$. 
If $S$ is a UFD, then $S$ is a regular local ring.   Moreover, 
either $S =  \ord_{R_i}$ for some base point  $R_i$  of $I$,   
or $S$ is an infinitely near point to $R$.\footnote{By definition, an infinitely near point has dimension at least two.} 
\end{theorem}

\begin{proof}
By Lemma~\ref{2.40}(1), $\m S$ is principal, so there exists a unique regular local ring 
$R'$ in the first neighborhood of $R$ such that $S$ dominates $R'$.
If $IR'$  is principal, then $R'$ dominates a unique local ring   on  $ \Proj \overline{R[It]}$.   
Since $S$ dominates $R'$, we have 
$R' = S$.  If $\dim R' = 1$, then $R' = \ord_R = S$  while if $\dim R' \ge 2$,  
then $R' = S$ is an infinitely near point to $R$ in the first neighborhood. 

If $IR'$ is not principal, then the transform $I' := I^{R'}$ is a proper ideal, so $R'$ is a base point of $I$ in the first neighborhood of $R$.
It follows that $I'$ is primary for the maximal ideal $\m'$ of $R'$ \cite[Prop. 1.21]{L} and is finitely supported with $\mathcal {BP}(I')$ a proper subset of $\mathcal {BP}(I)$.  Also   $S$  dominates $R'$ 
and $S  \in \Proj \overline{R'[I't]}$.
The assertions in  Theorem~\ref{2.13} therefore follow by a
straightforward induction argument on the number of base points of $I$.     
\end{proof}

\begin{discussion}  \label{1.32}
Let $I$ be a finitely supported  $\m$-primary  ideal of a regular local domain $(R,\m)$ with $\dim R  = d \ge 2$.  
Let $\Gamma :=  \mathcal {BP}(I) = \{R = R_0, R_1, \ldots, R_n\}$ denote the base points of $I$. 
 We observe the following:
\begin{enumerate}
\item
The set $\Rees I$ of Rees valuation rings of $I$ is a  nonempty  subset of the set $\{\ord_{R_i}\}_{i =0}^n$, cf. \cite[Prop.~4.3]{HK1}.     
\item
The DVRs $V \in \Proj\overline{R[It]}$ that dominate $R$ are precisely the DVRs $V \in \Rees I$ \cite[Theorem 10.2.2(3)]{SH}.
\item
 The DVRs   in  $ X_\Gamma $ that dominate $R$ are precisely the DVRs   in the set  $\{\ord_{R_i}\}_{i =0}^n$ (see Discussion~\ref{1.6}).
\item
Items 2 and 3 imply that the morphism $f: X_\Gamma  \to \Proj\overline{R[It]}$ of  Diagram~\ref{eq1.1} is an isomorphism on dimension one  local rings if and only if  $\ord_{R_i} \in \Rees I$ 
for each base point $R_i$ of $I$.
\end{enumerate}
Therefore  if  $f: X_\Gamma \to \Proj\overline{R[It]}$
is an isomorphism, then  the order valuation ring
of each base point of $I$ is a Rees valuation ring of $I$.
\end{discussion}  

\begin{remark}  \label{1.35}  Let $I$ be a finitely supported  $\m$-primary  ideal of a regular local 
domain $(R,\m)$ and let 
$S \in \Proj \overline{R[It]}$ be a local domain of dimension at least 2   that dominates $R$.  
Since $S \in \Proj \overline{R[It]}$, 
the  ring 
$S$  is not a base point of $I$.  Let $f:X_\Gamma \to \Proj \overline{R[It]}$ be the morphism  in Diagram~\ref{1.1}.
The local domains  $T \in X_\Gamma$ such that $f(T) = S$ are regular local 
domains. Each such $T$ is either  infinitely near to $R$ or $T = \ord_{R_i}$ for some base point $R_i$ of $I$. 
For each $T$ there exist 
injective local homomorphisms $R \hookrightarrow S \hookrightarrow T$.   Since  $T$  is 
birational over $S$,  we have $\dim S \ge \dim T$  \cite[Theorem~15.5]{M}.
If $S \ne T$ and $\dim S = \dim T$,  then  
  Zariski's Main Theorem \cite[(37.4)]{N}  implies  that $\m_S T$ 
is not primary for the maximal ideal of $T$,  where $\m_S$ is the 
maximal ideal of $S$.   Hence there exists a nonmaximal prime ideal $P$ of $T$ 
such that $\m_S \subset P$, and  $T_P \in X_\Gamma$ is a  regular  local ring in the fiber of $f$ over $S$ with  
$\dim T_P < \dim S$.  If $\dim S = 2$, then $T_P$ is a DVR in $X_\Gamma$ that dominates $R$.  This implies
that $T_P = \ord_{R_i}$, where $R_i$ is one of the finitely many  base points of $I$. 
Thus we have
\begin{enumerate}
\item  
Let $\Gamma :=  \mathcal {BP}(I) = \{R = R_0, R_1, \ldots, R_n\}$ denote the base points of $I$. 
If,  as in Discussion~\ref{1.32}.4,   we have 
 $\ord_{R_i} \in \Proj \overline{R[It]}$  for each $R_i \in \mathcal{BP}(I)$,  then the morphism 
$f:  X_{\Gamma} \to \Proj \overline{R[It]}$ is biregular at each $S \in \Proj \overline{R[It]}$  with $\dim S \le 2$  and  the singular locus of $\Proj\overline{R[It]}$ has codimension at least 3.  
\item 
If $\dim R = 3$  and each $\ord_{R_i} $ is a Rees valuation ring of $I$, then 
$\Proj \overline{R[It]}$ has only finitely many singular points.  
\end{enumerate}
 
 Without assuming that each $\ord_{R_i}$ is a Rees valuation of $I$,  let  
 $S \in \Proj \overline{R[It]}$ be a local domain of dimension at least 2   that dominates $R$.  Then  
$S \in X_\Gamma$ if and only if $S$ is infinitely near to $R$.   This follows because 
each local domain   $T \in X_\Gamma$ that dominates $R$ is 
infinitely near to $R$ and if $T$ dominates $S$,  then the unique finite sequence of local quadratic transforms 
of $R$ to $T$ goes through $S$ if and only if $S$ is infinitely near to $R$.  Thus if $S$ is infinitely near to $R$
and $T \in X_\Gamma$ dominates $S$, then either $S = T$ or $T \neq S$ is infinitely near to $S$.  
But if $T$ 
is infinitely near to $S$,  then $S$ must be one of the base points $R_i$,  which it is not since
 $S \in \Proj \overline{R[It]}$. 

We conclude that   $f:X_\Gamma \to \Proj \overline{R[It]}$ is 
biregular at $S$ for every  local ring $S \in \Proj \overline{R[It]}$ that is infinitely near to $R$, and 
the following  are equivalent:
\begin{enumerate}
\item [(i)]   The morphism $f: X_\Gamma \to \Proj \overline{R[It]}$ is an isomorphism.

\item [(ii)]   Each
 local ring $S$ on $\Proj \overline{R[It]}$ with $\dim S \ge 2$  that dominates $R$ is an infinitely near point to $R$. 
\end{enumerate}
These equivalent conditions imply that 
$\Proj \overline{R[It]}$ is regular.  
\end{remark}

\begin{corollary}  \label{2.14}
Let   $(R,\m)$ be a regular local domain with $\dim R \ge 2$ and let 
 $I$ be a finitely supported  $\m$-primary ideal.   If each local ring $S \in  \Proj \overline{R[It]}$
 is a unique factorization domain, then $\Proj \overline{R[It]}$ is regular and $I$ has a 
saturated factorization.  
\end{corollary}  

\begin{proof}  
By Theorem~\ref{2.13} every local ring $S \in \Proj \overline{R[It]}$ that dominates $R$ is either $\ord_{R_i}$ for a
base point $R_i$ of $I$ or is an infinitely near point to $R$.    By Remarks~\ref{1.35}, the morphism
$f: X_\Gamma \to \Proj \overline{R[It]}$ is an isomorphism. Thus $I$  has a saturated factorization.
\end{proof}

Theorem~\ref{2.13} and Discussion~\ref{1.32}  also imply the following.

\begin{corollary}  \label{2.15}
Let   $(R,\m)$ be a regular local domain with $\dim R \ge 2$ and let 
 $I$ be a finitely supported  $\m$-primary ideal.  
  \begin{enumerate}
\item
 If $R_i \in \mathcal{BP} (I)$ is such that  
  $\ord_{R_i}   \not\in  \Rees I$, 
  then for each regular local ring $T$ on $X_{\Gamma}$ such that $T \subseteq \ord_{R_i}$, the local ring on $\Proj \overline{R[It]}$ dominated by $T$ is not a UFD.
  Thus the local rings on $\Proj \overline{R[It]}$ dominated by $\ord_{R_i}$ or dominated by an infinitely near point in the first neighborhood of $R_i$ on $X_{\Gamma}$ are not UFDs.
\item
  If $R_i \in \mathcal{BP} (I)$ is such that $ord_{R_i} \not\in \Rees I$,
  and if $S \in \Proj \overline{R [I t]}$ is such that $S \subset \ord_{R_i}$, then $S$ is singular.
\item 
If $\Proj \overline{R [I t]}$ is regular,  then $\ord_{R_i} \in \Rees I$  for 
each $R_i \in \mathcal{BP}(I)$.
\end{enumerate} 
 \end{corollary}
 
\begin{proof}
The statement about $\ord_{R_i}$ in Item~1 follows directly from Theorem~\ref{2.13}.
For a local ring $T$ on $X_{\Gamma}$ such that $T \subset \ord_{R_i}$, let $S$ denote the local ring on $\Proj \overline{R[I t]}$ dominated by $T$.
The localization of $S$ at the center of $\ord_{R_i}$ is equal to the local ring on $\Proj \overline{R[I t]}$ dominated by $\ord_{R_i}$.
Thus a localization of $S$ is not a UFD, so $S$ is not a UFD.

Items 2 and 3 follow directly from Item 1.
\end{proof}

\section{Torsion divisor class group on normalized blowups} \label{sec3}

Let   $(R,\m)$ be a regular local domain with $\dim R \ge 2$ and let 
 $I$ be a finitely supported  $\m$-primary ideal. Let 
$S$  be a local ring on   $ \Proj \overline{R[It]}$.  In view of Theorem~\ref{2.13} it is natural to ask if
$S$ having torsion divisor class group  implies $ S$ is  regular.  This fails in general as we demonstrate
in Example~\ref{exam two bps}.   With additional assumptions  about the Rees valuation rings of $I$,  we show 
in Theorem~\ref{3.10} that if $S \in \Proj \overline{R[It]}$  has torsion divisor class group, then $S$ is 
regular. We use terminology as in Definition~\ref{3.09}

\begin{definition}  \label{3.09}
Let $A$ be an integral domain and let $B$ be an overring of $A$ with the same field of factions. 
The overring $B$ is a {\it sublocalization} of $A$ if $B$ is an intersection of localizations of $A$. 
Thus $B$ is a sublocalization of $A$ if and only if there exists a family 
$\{S_\lambda\}_{\lambda \in \Lambda}$ of multiplicatively closed subsets of nonzero elements of $A$ 
such that $B = \cap_{\lambda \in \Lambda}A_{S_\lambda}$.  It is well known that a sublocalization 
$B$ of $A$ is an intersection of localizations of $A$ at prime ideals.  Indeed, for a family 
$\{S_\lambda\}_{\lambda \in \Lambda}$ of multiplicatively closed subsets of nonzero elements of 
$A$, we have 
$$ \bigcap_{\lambda \in \Lambda}A_{S_\lambda} ~= ~
 \bigcap\{A_P ~|~ P \in \Spec A \text{ and } P \cap S_\lambda = \emptyset \text{ for some } \lambda \in \Lambda\}.
 $$
\end{definition}  

\begin{theorem} \label{3.10}
Let   $(R,\m)$ be a regular local domain with $\dim R \ge 2$ and let 
 $I$ be a finitely supported  $\m$-primary ideal. Let $\Gamma := \mathcal{BP}(I)$  denote
 the set of base points of $I$.
Let $Y$ be a normal complete model over $\Spec R$ that makes the following diagram commute, where $f: X_{\Gamma} \rightarrow \Proj \overline{R [I t]}$ is as in Diagram~\ref{1.1}:
\begin{equation*}
\xymatrix{
  X_{\Gamma}   \ar[rr]^{f}    \ar[dr]^{g} & & \Proj \overline{R[It]} \\
 & Y \ar[ur]^{h} &
}
\end{equation*}

If $ \ord_{R_i} \in \Rees I$ for each  $R_i \in \Gamma$, then we have
    \begin{enumerate}
 \item For each  local domain $S \in Y$ and 
 each $T \in g^{-1}(S)$, the ring $T$ is a sublocalization over $S$.
 \item   The morphism $g : X_{\Gamma} \rightarrow Y$
 is biregular at  each  $S \in Y$    that 
 has torsion divisor class group. 
 \item If $S \in Y$ is not regular, then the divisor class group of $S$ is not a torsion group.
 \end{enumerate}
\end{theorem}

\begin{proof}
Let $T \in g^{-1}(S)$ and let $A = h (S)$ be the local ring on $\Proj \overline{R [I t]}$ dominated by $S$.
We have injective birational local homomorphisms $A \hookrightarrow S \hookrightarrow T$
of normal Noetherian local  domains. 
We prove that $T$ is a sublocalization of $S$.  Since $S$ and $T$ are normal Noetherian 
domains,  it suffices to show that  
$T_{\p} = S_{ \p \cap S }$   for each  height one prime $\p$ of $T$.
By construction of $X_{\Gamma}$, either $T_{\p} = R_{\p \cap R}$ or $T_{\p} = V_i$
 for some $V_i = \ord_{R_i}$, where $R_i \in \Gamma$. 
In the case where $T_{\p} = R_{\p \cap R}$, it follows that $T_{\p} = S_{\p \cap S}$.
In the case where $T_{\p} = V_i$, let $\m_{V_i}$ denote the maximal ideal of $V_i$.
Since $V_i$ is a Rees valuation ring of $I$, it follows that $A_{\m_{V_i} \cap A} = V_i$.
Thus $V_i = A_{\m_{V_i} \cap A} \subseteq S_{\m_{V_i} \cap S} \subseteq T_{\p} = V_i$.
Noting that $\m_{V_i} \cap S = \p \cap S$, it follows that $\p \cap S$ is a height $1$ prime of $S$. 
Therefore $T$ is a sublocalization of $S$.  This  proves item~1.  
 
 If $S$ has torsion divisor class group, then every sublocalization of $S$ is a localization 
 of $S$, cf. \cite[Cor.~2.9]{HR}.   Since
 $S$ and $T$ are local and $S \hookrightarrow T$ is a local homomorphism,
 if $T$ is a localization of $S$, then $S = T$.  This proves item~2.  

Item~3 is the contrapositive of Item~2.
\end{proof}

\begin{corollary}
Assume the notation of Theorem~\ref{3.10}.
If $S \in \Proj \overline{R [I t]}$ is contained in $\ord_{R_i}$ for at most one $i$, then $f$ is biregular at $S$.
\end{corollary}

\begin{proof}
If $S$ is not contained in any $\ord_{R_i}$, then $S$ is a localization of $R$ and there is nothing to show, so assume $S$ is contained in $\ord_{R_i}$ for some fixed $i$.
It follows that $I S$ is principal, say $I S = a S$, and $a S$ has only one minimal prime $\p$, where $\p$ is the center of $\ord_{R_i}$.
Therefore $a S = \p^{(n)}$ for some positive integer $n$.
Since $S [\frac{1}{a}]$ is a localization of $R$, it is a UFD, so the divisor class group of $S$ is generated by the classes of minimal primes of $a S$.
Therefore the divisor class group of $S$ is torsion, so the claim follows from Theorem~\ref{3.10}.2.
\end{proof}

\begin{remark}   Lemma~\ref{2.40} can be used to give an alternative proof of item~2 of Theorem~\ref{3.10}
by an argument along the same lines as the proof given for Theorem~\ref{2.13}.
\end{remark}

\begin{discussion}
Let   $(R,\m)$ be a regular local domain with $\dim R \ge 2$ and let 
 $I$ be a finitely supported  $\m$-primary ideal.  
 Let $\Gamma := \mathcal{BP}(I)$  denote
 the set of base points of $I$ and  let $V_i := \ord_{R_i} $ for each  $R_i \in \Gamma$.
 In view of Theorem~\ref{3.10},  we are motivated to ask for conditions on $I$ that imply
 $\ord_{R_i}  \in \Rees I$ for each $R_i \in \Gamma$.   Lipman's unique factorization 
 theorem for finitely supported complete ideals implies that  $I$ has a factorization as a
 product of special $*$-simple complete ideals with possibly some negative exponents.  
 For each terminal base point $R_n$ the special $*$-simple ideal $P_{RR_n}$ must occur
 with a positive exponent.  Since $V_n \in \Rees P_{RR_N}$, it follows that 
 $V_n = \ord_{R_n} \in \Rees I$ for each 
 terminal base point $R_n$ of $I$. For each  $R_j \in \Gamma$,  the DVR 
 $V_j = \ord_{R_j}$   
 dominates a unique local domain $S_j \in \Proj \overline{R[It]}$.  Theorem~\ref{2.13} implies
 that   $S_j = V_j$      if  $S_j$ is a UFD.  Hence $V_j \in \Rees I$ in this case.

 In Example~\ref{samerees}, we 
 present an example where  $\ord_{R_i}  \in \Rees I$ for each $R_i \in \Gamma$
  and $\Proj \overline{R[It]}$ has
 precisely one singular point.  
 
 Fix a local domain  $S \in \Proj \overline{R[It]}$  that dominates $R$.   We observe the following:
 \begin{enumerate}
 \item 
 Since $I$ is $\m$-primary, the minimal primes $P$ of $IS$ are
 the same as the minimal primes of $\m S$.   Since $IS$ is principal,  each minimal prime $P$ of $IS$ has $\hgt P = 1$. 
 Since $S$ is normal, $S_P = V$ is a DVR and $V \in \Rees I$.  Thus the association of a minimal prime 
 $P$ of $IS$ or $\m S$  with the localization $S_P = V$ yields  a 
 one-to-one correspondence  between the minimal primes $P$ of $I$  and 
 the DVRs $V \in \Rees I$ such that $V$ contains $S$.  
 \item
 Let $f : X_{\Gamma} \rightarrow \Proj \overline{R [I t]}$  be as in Diagram~\ref{1.1}.  The morphism $f$ is either biregular at 
 $S$ or the fiber $f^{-1}(S)$ is infinite and contains both local domains $T$ with $\dim T = \dim S$ and local domains $T$
 with $\dim T < \dim S$.  To see that there exists $T \in f^{-1}(S)$ with $\dim T = \dim S$,  let 
 $(0) = \p_0 \subset \p_1 \subset \cdots \subset \m_S$  be a strictly ascending chain of prime ideals of $S$ of length equal 
 to $\dim S$.
  By \cite[(11.9)]{N}, there exists a valuation domain $W$ that has prime ideals lying over each prime ideal in 
 this chain. Let $T$ be the local ring on $X_\Gamma$ dominated by $W$.  Then $T \in f^{-1}(S)$ and 
 we have $\dim T \ge \dim S$ since $T$ contains a chain of prime ideals of length $\dim S$  that contract in $S$ to distinct 
 prime ideals.  Since $S$ is Noetherian, we also have $\dim T \le \dim S$, so $\dim T = \dim S$. 
 
 Assume that $S \ne T$ and let $I^*$ be an ideal in $R$ such that  $I^*$ has a saturated factorization and
  $\mathcal{BP}(I^*) = \mathcal{BP}(I)$.  
  Let $a \in I^*$ be such that $aT = I^*T$ and let $A := S[I^*/a]$.   We have
 $S \hookrightarrow A \hookrightarrow T$  and $T$ is a localization of $A$ at a maximal ideal 
 $P$,  where
 $P \cap S = \m_S$  is  the maximal ideal of $S$.   Since $R[I^*/a] \subset S[I^*/a]$, 
 we have $A_Q \in X_\Gamma$   for each $Q \in \Spec A$.   As in Remark~\ref{1.35},  
 the ideal $\m_ST$  is contained in a 
 nonmaximal prime ideal of $T$. Hence there exists a prime ideal $Q$ of $A$ such that $Q \cap S = \m_S$
 and $Q \subsetneq P$.  Thus the ring $A/\m_SA$  has positive Krull dimension,  and 
  is a finitely generated algebra over the residue field of $S$.  Therefore $\Spec (A/ {\m_S} A)$ is 
  infinite and hence the fiber $f^{-1}(S)$ is infinite.
  \item
 Since $S$ is  a normal local domain,  $S$ is the intersection of the valuation domains $W$ that birationally 
 dominate $S$, cf \cite[Prop. 1.1]{L}. Each of these valuation domains $W$  dominates a regular
 local domain $T \in f^{-1}(S)$.  It follows that  $S = \bigcap\{T ~|~ T \in f^{-1}(S) \}$.  
  \end{enumerate}
\end{discussion}

\section{Ideals that have a saturated factorization}   \label{sec4}

\begin{discussion}\label{same blowup}  Let $(R,\m)$ be a regular local ring and let  $I$ be a finitely supported complete  
$\m$-primary ideal.  Let $\mathcal{BP}(I)$ be the base points of $I$ and 
enumerate the base points as $R = R_0, R_1, \ldots, R_n$.  
For $i \in \{ 0,\dots, n \}$, let $P_i$ denote the special $*$-simple ideal of $R$ associated to the pair $R \prec R_i$. 
We consider the following properties the ideal $I$ may have.  Each of the enumerated  properties implies that 
$I$ has a saturated factorization, that is 
$\Proj \overline{R[It]}$ is regular and is equal  to $X_{\Gamma}$.
\begin{enumerate}
\item The product $P_0 *P_1* \cdots *P_n$ divides $I$  in the sense that there exists an ideal $J$ of $R$ such that 
$P_0 *P_1* \cdots *P_n*J = I$.  
\item 
For each  $R_i \in \mathcal{BP}(I)$,  the special star-simple ideal $P_i$ divides $I$.  
\item 
The product $P_0 *P_1* \cdots *P_n$ divides $I^k$, for some positive integer $k$.
\item 
For each $S \in \mathcal{BP}(I)$,   the complete transform 
$\overline{I^S}$  of $I$ in $S$ is
divisible by the maximal ideal $\m_S$ of $S$,  that is  $ \overline{I^S} = \m_S * J$ for some ideal $J \subset S$.
\item 
There exists a positive integer $k$ such that for   each $S \in \mathcal{BP}(I)$,   the complete transform 
$\overline{(I^k)^{S}}$  of $I^k$ in $S$ is
divisible by the maximal ideal $\m_S$ of $S$,  that is  $\overline{(I^k)^{S}}  = \m_S * J$ for some ideal $J \subset S$.
\end{enumerate}

It is straightforward  to see that $(1) \Rightarrow (2) \Rightarrow (3) \Rightarrow (5)$ and $(4) \Rightarrow (5)$.
Since the ideals $I$ and $I^k$ have the same normalized blowup, and 
since complete transforms and $*$-products commute,  Condition 5 implies
that $\Proj \overline{R[It]} = X_{\Gamma}$. 
\end{discussion}

 Example \ref{exam:twoBranches}   demonstrates  the existence of a 
finitely supported complete ideal   of a regular local domain  that satisfies 
Condition~2 but fails to satisfy Condition~1 of Discussion \ref{same blowup}.

\begin{example}\label{exam:twoBranches}
Let $R$ be a $3$-dimensional regular local ring with maximal ideal $\m = (x,y,z)R$. Consider the following infinitely near points $R_i$ of $R$:
\begin{equation*}
  R~ := ~R_0
    \begin{array}{cc}
      \prec^x~ R_1 & \prec^{xy} ~R_2\\
      \prec^z ~R_3 & \prec^{zy}~ R_4
    \end{array}.
\end{equation*}  
Thus $R_1$ and $R_3$ are in the first neighborhood of $R$ and $R_2$ and $R_4$ are in the second neighborhood of $R$.

The special $*$-simple ideals $P_i$ associated to the pairs $R \prec R_i$ are 
\begin{align*}
P_0 &= \m,\\
P_1 &= (x^2, ~ y,  ~z)R, \\
P_2 &= (x^3,~ x^2y,~ xz,~ y^2, ~ yz,~ z^2)R, \\
P_3 &= (z^2, ~x, ~y)R, \\
P_4 &= (z^3, ~z^2y, ~zx, ~  y^2,   ~ yx,  ~x^2)R.
\end{align*}

The product  $P_2 * P_4  = P_2P_4$  is divisible by $\m^2$ and  has a factorization $J * \m^2   = J\m^2$,   where 
$J :=  (x z,y^{2},z^{3},y z^{2},x^{2} y,x^{3})R$.  
By an argument similar to \cite[Example~4.18]{HKT}, the ideal $J$ is a $*$-simple ideal that is not a special $*$-simple ideal. The ideal 
$J$ has two Rees valuations,  $\Rees J = \{\ord_{R_2}, \ord_{R_4}\}$,  
the order valuations of $R_2$ and $R_4$.
Consider the ideal 
\begin{align*}
I &:= J * \m * P_1 * P_3    = J\m P_1P_3 \\
 &=  (x y z^{3},x^{2} z^{3},y^{3} z^{2},x y^{2} z^{2},x^{2} y z^{2},x^{3} z^{2},y^{4} z,x y^{3} z,x^{2} y^{2} z,x^{3} y z,y^{5},x y^{4},\\
 &\qquad x^{2}y^{3},y z^{5},x z^{5},y^{2} z^{4},x^{5} z,x^{4} y^{2},x^{5} y,z^{7},x^{7})R.
\end{align*}
Each of the  ideals   $P_2$ and $P_4$ divides  $I$,  so $I$ satisfies Condition~2 of 
Discussion~\ref{same blowup}.  
Since $\ord_R (P_1 * P_2 * P_3 * P_4)  = 6  >  \ord_R I = 5$, 
the $*$-product $P_1 * P_2 * P_3 * P_4$ does not divide $I$. 
Hence,   a fortiori,   $\m * P_1 * P_2 * P_3 * P_4$ does not divide $I$, 
 so the ideal $I$  does not satisfy Condition~1 of 
Discussion \ref{same blowup}.
\end{example}

In Example~\ref{7.10},  we examine singularities of the  $ *$-simple  monomial ideal $J$ 
of Example~\ref{exam:twoBranches}.

\section{Blowups of ideals with only two  base points}

We consider in this section the case where  a finitely supported ideal has two base points 
and no residue field extension.

\begin{setting}  \label{5.1}
Let $(R, \m)$ be a  regular local  domain  with $d = \dim R \ge 2$ and let $R_1$ be an infinitely near 
point to $R$ in the first neighborhood. Assume there is no residue field extension from $R$ to 
$R_1$.  By appropriately choosing a regular system of parameters for $R$,  we may assume that 

$\m = (x,~y_1,~\ldots,~y_{d-1})R$ \, and \, $R_1 = R[\frac{y_1}{x}, \ldots, \frac{y_{d-1}}{x}]_{(x, \frac{y_1}{x}, \ldots, \frac{y_{d-1}}{x})R[\frac{y_1}{x}, \ldots, \frac{y_{d-1}}{x}]}$.

The special $*$-simple ideal associated to $R$ as an infinitely near point to itself is the maximal ideal  $\m$ of $R$.
The special $*$-simple ideal associated to $R \prec R_1$ is
$P_1 ~ = ~(x^2, y_1, \ldots,  y_{d-1})R$.

\end{setting}

\begin{discussion}
With notation as in Setting~\ref{5.1}, let $\Gamma := \{R_0, R_1\}$.  
For finitely supported ideals $I$ with $ \mathcal{BP}(I) = \Gamma$, we
observe that there are precisely two possibilities for the model 
$\Proj \overline{R[It]}$.  By   \cite[Theorem~5.4]{HKT} and 
the unique factorization theorem of Lipman  \cite[Theorem~2.5]{L},   the complete ideals $I$ 
such that $\mathcal{BP}(I) =  \{R_0, \, R_1\}$  have the form 
$I = \m^i * P_{1}^j$, where $i$ is a nonnegative 
integer and $j$ is a positive integer. 
\begin{enumerate}
\item
Assume in the  factorization $I = \m^i * P_{1}^j$ that 
the integer $i$ is positive.   Then   $I$ has a saturated factorization and  $\Proj R[It] = X_{\Gamma}$.  
We may take $i = j = 1$.  
The model $X_\Gamma$ has infinitely near points to $R$ in the first   and  second  neighborhoods.  
\item
The ideals $P_{1}^j$ for $j$ a positive integer all have the same blowup.
Thus one may assume $j = 1$.  We examine the model $\Proj R[P_1t]$   in Example~\ref{exam two bps}.  
\end{enumerate}

\end{discussion}

\begin{example}\label{exam two bps}
Assume notation as in  Setting~\ref{5.1}.  The order valuation domain $\ord_R$ is not in $\Rees P_1$.
  By  Corollary~\ref{2.15},  the local ring $S \in \Proj R[P_1t]$   dominated by 
$\ord_R$ is not a UFD.  We show the following.

\noindent
{\bf Fact.}  The morphism $f: X_{\Gamma} \to \Proj R[P_1t]$ as in  Diagram~\ref{eq1.1} is biregular at the local rings 
$T$ on $\Proj R[P_1 t]$ such that 
$S$  is not a localization  of $T$, i.e. such that $T \not\subset \ord_R$.   
Using the language of schemes, let $\mathfrak{p} \in \Proj R[ P_1 t]$ be the point corresponding to the local 
domain 
$S$,  that is, $\SO_{\Proj R[ P_1 t], \mathfrak{p} }  = S$. Then $f$ induces an isomorphism on the open sets   
  $$X_\Gamma ~\setminus  ~f^{-1} ( \overline{\{\mathfrak p\}} ) ~\to ~
  \Proj R[P_1 t] ~ \setminus ~ \overline{\{\mathfrak p\}},
 $$
 where $\overline { \{ \mathfrak{p} \} }$ denotes the Zariski closure of the point $\mathfrak{p}$. 

To establish the fact stated above,  let 
$T \in \Proj R[P_1 t]$ be a local ring birationally dominating $R$ such that $T \not\subset \ord_{R}$.
We show that $f$ is biregular at $T$.
Let $\q$ denote the center of $\ord_{R_1}$ on $T$ and let $P_1 T = a T = \q^{(2)}$.
For each height $1$ prime ideal $\p$ of $T$,
we have $T_{\p}$ is either $R_{\p \cap R}$ or $\ord_{R_1}$.
Since $T [\frac{1}{a}]$ is a Noetherian normal domain and $\q$ is the unique minimal prime ideal of $a T$, the ring $T [\frac{1}{a}]$ is a sublocalization of $R$.
Since $R$ is a UFD, it follows that $T [\frac{1}{a}]$ is also a localization of $R$ (\cite[Cor 2.9]{HR}).
Thus $T [\frac{1}{a}]$ is a UFD and the divisor class group of $T$ is generated by the divisor class of $\q$.
Since $\q^{(2)} = a T$ is principal, it follows that the divisor class group    of $T$ is  a torsion group.
By Lemma~\ref{2.40}, $\m T$ is principal.  Since $\Proj R[\m P_1] = X_\Gamma$,  it follows that 
$T$ is on $X_{\Gamma}$, so $T$ is regular and $f$ is biregular at $T$.

We conclude  that the fiber with respect to $f: X_\Gamma \to \Proj R[P_1t]$
 of the singular locus of $\Proj R [P_1 t]$ consists of the rings $T$ on $X_{\Gamma}$ such that $T \subset \ord_{R}$.
In particular, every point in the first neighborhood of $R$ except $R_1$ is in the fiber of the singular locus of $\Proj R [P_1 t]$.
To see this, let $\mathfrak{q}$ denote the point corresponding to $\ord_{R}$ in $\Proj R[P_1 t]$.
Then $f^{-1} (\overline{ \{ \mathfrak{p} \} }) = \overline{ \{ \mathfrak{q} \} }$.

In the case where $\dim R = 2$,   the local domain  $S$ is the unique singular point  of $\Proj R[P_1t]$.  
The fiber  $f^{-1}(S)$ consists of the infinitely near points in  the first neighborhood of $R$ other than $R_1$
 and the point $R [\frac{x^2}{y}, \frac{y}{x}]_{(\frac{x^2}{y}, \frac{y}{x}) R [\frac{x^2}{y}, \frac{y}{x}]}$.
 Notice that $R [\frac{x^2}{y}, \frac{y}{x}]_{(\frac{x^2}{y}, \frac{y}{x}) R [\frac{x^2}{y}, \frac{y}{x}]}$  is 
 the unique point in the first neighborhood of $R_1$ that is contained in $\ord_{R_0}$.
 In classical terminology, this point is said to be proximate to $R_0$.
 
 In the case where $\dim R   = n \ge 3$,  the local domain $S$ is no longer the unique singular point of $\Proj R[P_1t]$.
 We have $\dim S = 2$,  and the singular locus of $\Proj R[P_1t]$ is  of dimension $n - 2$. 
\end{example}

In Example~\ref{exam two bps},  the powers of the maximal ideal of the local domain 
$S$ define a valuation ring $\ord_S$  and $\ord_S = \ord_{R_0}$.
This motivates us to ask:

\begin{question}\label{5.4}
Let $(R, \m)$ be a  regular local ring   with $\dim R \ge 3$   and let $I$ be a finitely supported 
$\m$-primary ideal of $R$.
Let $R'$ be a base point of $I$ such that $V = \ord_{R'}$ is not a Rees valuation ring of $I$.
Let $(S, \m_S)$ denote the ring birationally dominated by $V$ on $\Proj \overline{R [ I t ]}$.  If  the powers of 
$\m_S$ define a valuation  ring  $\ord_S$,  does it  follow that $\ord_S = V$? 
\end{question}

\begin{remark}
Let  $R \hookrightarrow S$ be an injective  extension of regular local domains  with $\dim R = \dim S$ and $S$ birationally dominating $R$.
If $\ord_{R} = \ord_{S}$, then  it  follows from \cite[Cor.~2.6]{Sally} that $R = S$.
\end{remark}

Proposition~\ref{5.6} answers Question~\ref{5.4} in the case where $V = \ord_R$.

\begin{proposition}\label{5.6}
Let $R$ be a Noetherian  local domain such that the powers of its maximal ideal $\m_R$ define a 
valuation.  Let $V = \ord_R$ denote the associated valuation domain.
Let $S$ be a local domain birationally dominating $R$ such that  $V$ dominates $S$.
If the powers of the maximal ideal $\m_S$ of $S$ define a valuation,  then $V$ is the order
valuation ring $\ord_S$. 
\end{proposition}

\begin{proof}
Let $a \in R$.
Since $S$ dominates $R$, $\ord_{R} a \le \ord_{S} a$, and since $V$ dominates $S$, we have $\ord_{S} a \le \ord_{R} a$, so $\ord_{R} a = \ord_{S} a$.
Thus $\ord_{R} = \ord_{S}$ on their common field of fractions, so $V$ is the order valuation ring $\ord_{S}$.
\end{proof}

\section{Finitely supported ideals having  the  same Rees valuations}

The examples   we  present in this   section have  3 base points 
with the base points linearly ordered.   We describe  the blowups of all the complete ideals having 
precisely these 3 points as 
base points.

Steven Dale Cutkosky remarks  in \cite{C1}  that a birational morphism between 2-dimensional normal schemes that is an
isomorphism in codimension one must be an isomorphism by Zariski's main theorem.   In Example~2 on 
page~37 of \cite{C1},  Cutkosky presents an example of an infinite set of normal ideals in a 3-dimensional 
regular local ring that have the same Rees valuations, but   have the property 
that the blowups of the ideals are pairwise distinct.
In Example~\ref{samerees},  we present an example of
normal ideals  $J \subset I$   of a 3-dimensional regular local ring  $R$
that have the same Rees valuations, 
the ideal $J$ is a multiple of $I$ and 
$\Proj R[Jt] = X_{\Gamma}$ is regular while  $\Proj R[It]$ has one singular point.

\begin{setting}  \label{6.10}
Let $(R, \m)$ be a  regular local domain with $d = \dim R \ge 2$.  
Let $\m = (x, y)R$ if $d = 2$ and $\m = (x, y, z_1, \ldots, z_{d-2})$ if $d \ge 3$ (and if $d = 3$, denote $z = z_1$).
Consider the following chain of local quadratic transforms
$$
R  ~~ := R_0 ~  \prec^{x} ~ R_1 ~ \prec^{yx} ~ R_2,
$$
where $R_1$ with maximal ideal $\m_1$ is as in Setting~\ref{5.1}. Thus \\
\noindent
$R_1 = R[\frac{y}{x}]_{(x, \frac{y}{x})R[\frac{y}{x}]}$  \quad   and \quad $\m_1 =  (x, \frac{y}{x}) R_1$
 \quad   if $d = 2$,\\
 \medskip
\noindent
$R_1 = R[\frac{\m}{x}]_{(x, \frac{y}{x}, \frac{z_1}{x}, \ldots, \frac{z_{d-2}}{x})R [\frac{\m}{x}]}$  \quad and \quad
 $\m_1 =  (x, \frac{y}{x}, \frac{z_1}{x}, \ldots,  \frac{z_{d-2}}{x}) R_1$
  \quad if $d \ge 3$.\\
\noindent
Then 
$$
    S_2~ := ~ R_1[\frac{\m_1}{y/x}]   \qquad \text{and} \qquad R_2 ~:=~ (S_2)_{N_2} \qquad  
$$
where $N_2 := (\frac{x^2}{y}, \frac{y}{x})S_2$~ if $d = 2$ and ~ $N_2 := (\frac{x^2}{y}, \frac{y}{x}, \frac{z_1}{y}, \ldots, \frac{z_{d-2}}{y})S_2$ if $d \ge 3$.

For $i \in \{ 0, 1, 2 \}$, let $P_i$ denote the special $*$-simple ideals associated to the  
extension $R_0 \prec R_i$. 
We list  generators for the ideals  $P_i$ and the values of the variables with respect 
to the order valuation rings $\ord_{R_i}$.  
 If $d = 2$, we have
$$
\begin{array}{lcl}
P_0 & = & (x,y)R = \m \\
P_1 & = & (x^2, y)R \\
P_2 & = & (x^3, x^2y, y^2)R,  
\end{array}
\qquad \qquad
\begin{array}{c|c|c}
	& x 	& y 	 \\ \hline
\ord_{R_0} & 1	& 1 	 \\ \hline
\ord_{R_1} & 1 & 2 	 \\ \hline
\ord_{R_2} & 2	& 3   \\ 
\end{array}.
$$
If $d \ge 3$, then 
$$
\begin{array}{lcl}
P_0 & = & (x,y,z_1, \ldots, z_{d-2})R = \m \\
P_1 & = & (x^2, y, z_1, \ldots, z_{d-2})R \\
P_2 & = & (x^3, x^2y, x (z_1, \ldots, z_{d-2}), (y, z_1, \ldots, z_{d-2})^2)R,  
\end{array}
\qquad \qquad
\begin{array}{c|c|c|c}
	& x 	& y 	& z_i \\ \hline
\ord_{R_0} & 1	& 1 	& 1 \\ \hline
\ord_{R_1} & 1 & 2 	& 2 \\ \hline
\ord_{R_2} & 2	& 3  & 4 \\ 
\end{array}.
$$

\noindent
As in \cite[Example 6.13]{HK1} or \cite[Cor.~5.9]{HKT},  if $d \ge 3$,  then  $\Rees  P_2  = \{ \ord_{R},\,  \ord_{R_2} \}$. 
If $d = 2$, then $\Rees P_2 = \{ \ord_{R_2} \}$ as in Discussion~\ref{1.32}.1.
\end{setting} 

\begin{discussion}   With notation as in Setting~\ref{6.10},  let $\Gamma := \{R_0, R_1, R_2\}$.  
For finitely supported ideals $I$ with $ \mathcal{BP}(I) = \Gamma$, we
observe that there are precisely 4 possibilities for the model 
$\Proj \overline{R[It]}$.  By   \cite[Theorem~5.4]{HKT} and 
the unique factorization theorem of Lipman  \cite[Theorem~2.5]{L},   the complete ideals $I$ 
such that $\mathcal{BP}(I) = \Gamma$  have the form 
$I = \m^i * P_{1}^j*P_{2}^k$, where $i$   and $j$ are  nonnegative 
integers   and $k$ is a positive integer.  There are the following 4 possible models
$\Proj \overline{R[It]}$.  
\begin{enumerate}
\item
Assume in the  factorization $I = \m^i * P_{1}^j*P_{2}^k$ that 
  $i$   and $j$ are both positive.   Then   $I$ has a saturated factorization, i.e.,  $\Proj R[It] = X_{\Gamma}$.  
We may take $i = j = k = 1$.  The ideal $\m*P_1*P_2 = \m P_1P_2$ gives the blowup.

\item
Assume in the  factorization $I = \m^i * P_{1}^j*P_{2}^k$ that $i > 0$ and $ j = 0$.
The ideals $\m^iP_{2}^k$ for $i$ and $k$ positive all have the same blowup. Thus we may assume $i = k = 1$. 
The ideal $\m*P_{2} = \m P_2$ gives this blowup.

\item 
Assume in the  factorization $I = \m^i * P_{1}^j*P_{2}^k$ that $i = j = 0$.
The ideals $P_{2}^k$ for $k$ a positive integer all have the same blowup.
Thus one may assume $k = 1$.  The ideal $P_2$ gives this blowup.

\item
Assume in the  factorization $I = \m^i * P_{1}^j*P_{2}^k$ that $i = 0$ and $ j > 0$.
The ideals $P^j_{1}*P_{2}^k$ with  $j$ and $k$ both  positive all have the same blowup. Thus we 
may assume $j = k = 1$.  The ideal $P_1*P_2 = P_1P_2$ gives this blowup.  
\end{enumerate}
\end{discussion}  

The four models and the  natural morphisms among these models are displayed in Diagram~\ref{eq5.20}.

\begin{equation*}  \label{eq5.20} \tag{6.2}
\xymatrix{
	& X_\Gamma  \ar[dl]_{f_{\m}} \ar[dr]^{f_{P_1}} & \\
\Proj R[P_1 P_2 t] \ar[dr]^{\phi_{P_1}} & 	& \Proj R[\m P_2 t] \ar[dl]_{\phi_{\m}} \\
	& \Proj R[P_2 t] \ar[d] \\
	& \Spec R &
}
\end{equation*}

There are significant differences between the case where $\dim R = 2$ and the case where $\dim R \ge 3$ 
that are related to the fact that $\Rees P_2 = \{\ord_{R_2} \}$ if  $\dim R = 2$ while  
$\Rees P_2 =   \{\ord_{R_2}, \ord_R \}$ if $\dim R \ge 3$.  In Example~\ref{6.25} we describe the 
situation where  $\dim R = 2$.

\begin{example}  \label{6.25}  Assume notation as in Setting~\ref{6.10} and that $\dim R = 2$. 
Thus $P_2 = (x^3, x^2y, y^2)R$ and $\Proj R[P_2t]$ has 2 singular points
$$
S_0 ~:= ~ R[\frac{x^3}{y^2}, \frac{x^2}{y}]_{(x, y, \frac{x^3}{y^2}, \frac{x^2}{y})R[\frac{x^3}{y^2}, \frac{x^2}{y}]} \quad
\text{and} \quad S_1 ~:=~ R[\frac{y}{x}, \frac{y^2}{x^3}]_{(x, \frac{y}{x}, \frac{y^2}{x^3})R[\frac{y}{x}, \frac{y^2}{x^3}]}.
$$
The local domain 
$S_0  \in  \Proj R[P_2t]$   is  dominated by $\ord_{R}$ and $S_1  \in \Proj R[P_2t]$
is  dominated by $\ord_{R_1}$.  The divisor class group $\Cl(S_0)$ is a cyclic group of order 3,  and 
the divisor class group $\Cl(S_1)$ is  a cyclic group of order 2.   The local domains $S_0$ and $S_1$ are 
both localizations of the affine chart $R[\frac{P_2}{x^3 + y^2}]$ and the divisor class group 
$\Cl(R[\frac{P_2}{x^3 + y^2}])$ is a cyclic group of order 6.  

The local domain $S_0$  is also on the 
model $\Proj R[P_1P_2t]$ and is the unique singular point on this model,  while the local domain $S_1$ is on the 
model $\Proj R[\m P_2t]$ and is the unique singular point on this model.   

With notation as in Diagram~\ref{eq5.20}, we have:
\begin{enumerate}
\item  The morphism $\phi_{P_1}$ is an isomorphism off the fiber 
$\phi_{P_1}^{-1}(S_1)$.
\item
The morphism $\phi_{\m}$ is an isomorphism off the fiber $\phi_{\m}^{-1}(S_0)$.
\item  
The morphism $f_{P_1}$ is an isomorphism off the fiber 
$f_{P_1}^{-1}(S_1)$.
\item
The morphism $f_{\m}$ is an isomorphism off the fiber $f_{\m}^{-1}(S_0)$. 
\end{enumerate}
This completes our description of the case where $\dim R = 2$.
\end{example}

Assume that $\dim R = 3$.
In Examples \ref{6.29}, \ref{6.30}, and \ref{samerees}, we consider the models obtained by blowing up the ideals $\m P_2$, $P_2$, and $P_1 P_2$, respectively.

\begin{example}\label{6.29}
Assume notation as in Setting~\ref{6.10} with $\dim R = 3$.
Consider the ideal $I = \m P_2$ and its blowup
$$
 \Proj R[It]  = \Proj  R[x^4t, x^3yt, x^2zt, xy^2t, xyzt, xz^2t, y^3t, y^2zt, yz^2t, z^3t].
$$
The transform of $I$ in $R_1$ is the ideal 
$$ 
I_1 : = ~ \left( x, ~ \left(\frac{y}{x}\right)^2, ~ \frac{z}{x} \right)R_1
$$
and $I_1$ is the special $*$-simple ideal associated to the pair $R_1 \prec R_2$, see \cite[Prop. 2.1]{L}.
The natural  morphism $\phi_{P_2} : \Proj R[I t] \rightarrow \Proj R [\m t]$ is an isomorphism off the fiber 
$\phi^{-1}_{P_2}(R_1)$
of $R_1$.  
  Moreover,  with $I_1$ the transform of $I$ in $R_1$,  
the restriction $\phi_{P_2} : \Proj R_1[I_1 t] \rightarrow \Spec R_1$  is as in Example~\ref{exam two bps}.
Thus the singular locus of $\Proj R[It]$ is determined by the center of $\ord_{R_1}$ on $\Proj R[It]$.
\end{example}

\begin{example}  \label{6.30}  
Assume notation as in Setting~\ref{6.10}  with $\dim R = 3$. 
Consider the ideal $P_2$, where
$$
 P_{2}~  = ~  (x^3,~ x^2y, ~ xz,~ y^2, ~ ~ yz, z^2)R  
$$
is the special $*$-simple ideal associated to the extension $R_0 \prec R_2$. 
The blowup of $P_2$ is 
$$
 \Proj R[P_2 t] = \Proj R[x^3t, x^2yt, xzt, y^2t, yzt, z^2t].  
$$

We consider  affine charts of $\Proj R [ P_2 t]$ and examine  their singularities. 
The ideal  $(y^2, xz,  z^2,  x^3)R$ is a monomial  reduction  of $P_2$.   
It suffices to consider  the affine charts 
$R [ \frac{P_2 }{ \rho } ]$, where $\rho \in \{ y^2, xz,  z^2,  x^3 \}$.  We have the four affine charts:
$$
A~ :=~ R[\frac{P_2}{y^2}] \qquad  B~:=  ~ R[\frac{P_2}{xz}] \qquad  C~:=~ R[\frac{P_2}{z^2}] 
\qquad D~ := ~ R[\frac{P_2}{x^3}].
$$

The affine chart $A  = R[ \frac{x^3}{y^2},  \frac{x^2}{y},  \frac{xz}{y^2}, \frac{z}{y} ]$ 
is not contained in $\ord_{R_1}$ since $\frac{x^3}{y^2}$ and $\frac{xz}{y^2}$ have negative value for $\ord_{R_1}$.
By individually inverting each of the generators of $\m_A :=  (x, y, \frac{z}{y}, \frac{xz}{y^2}, \frac{x^2}{y}, \frac{x^3}{y^2})A$ and checking that the ring we obtain is regular, we conclude that $\m_A$ is the unique singular point of $A$.
We compute that $S := A_{\m_A}$ is a 3-dimensional Cohen-Macaulay normal local domain of embedding dimension 6 and multiplicity 4 where, for instance, $(\frac{z}{y}, \frac{x^3}{y^2}, y - \frac{xz}{y^2}) S$ is a system of parameters for $S$.
The ring $A$ is also an affine chart for $\Proj R[P_1P_2t]$ and $S$ is the unique singular point of 
the model $\Proj R[P_1P_2t]$.   We examine this in more detail in Example~\ref{samerees}.

The affine chart 
$C =   R[\frac{P_2}{z^2}] =  R[\frac{x}{z}, \frac{y}{z}]$ is regular.   

The affine chart  $D = R[\frac{y}{x}, \frac{z}{x^2}, \frac{y^2}{x^3}]$ 
is contained in  the valuation domain $\ord_{R_1}$.   The center of $\ord_{R_1}$ on $D$ is the height 2 
prime ideal $Q := (x,  \frac{y}{x}, \frac{y^2}{x^3})D$.    We compute that $D_Q$ is a 2-dimensional normal 
local domain of multiplicity 2.  Moreover, the singular locus of $D$ is the set of prime ideals of $D$ that 
contain $Q$.    

The affine chart  $B = R[\frac{x^2}{z}, \frac{y^2}{xz},  \frac{y}{x}, \frac{z}{x}]$ 
is also  contained in $\ord_{R_1}$.  The   center of $\ord_{R_1}$ on $B$ is the height 2 prime ideal
$Q' := (\frac{y^2}{xz},  \frac{y}{x}, \frac{z}{x})B$.  We   have $B_{Q'} = D_Q$,  and 
 compute that the singular locus of $B$ is the set of 
prime  ideals of $B$ that contain $Q'$.  

Since $\m B$ and $\m D$ are principal, the affine charts $B$ and $D$ of $\Proj R [P_2 t]$ are 
also affine charts of $\Proj R [\m P_2 t]$ and the morphism  $\phi_{\m}$ of Diagram~\ref{eq5.20} 
is an isomorphism on these affine charts.

 The local domain  on $\Proj R[P_2t]$ dominated by $\ord_{R_1}$ is  $B_{Q'} = D_Q$.
 The morphism $\phi_{P_1}$   of Diagram~\ref{eq5.20}  is biregular  at all the local domains 
 $S \in \Proj R[P_2t]$ except those $S$ such that  $B_{Q'} = D_Q$ is a localization of $S$,
 that is,  the morphism $\phi_{P_1}$  is biregular  
 off the center of $\ord_{R_1}$ on 
$\Proj R[P_2t]$. 
\end{example}

\begin{example}  \label{samerees}
Assume notation as in Setting~\ref{6.10}    with $ \dim R = 3$  and let 
$$
I ~= ~ {P}_{1} {P}_{2} ~=~( z^3, ~yz^2, ~xz^2, ~y^2z, ~xyz, ~ y^3,  ~x^3z, ~x^2y^2, ~ x^3y, ~x^5 )R.
$$   
Let   $J := \m P_1 P_2 = \m I$. 
By Remark~\ref{Gamma} the ideal $J$ has a saturated factorization, i.e., $R[Jt] = X_{\Gamma}$. 
We  have  $\Rees I = \Rees J   = \{ \ord_{R_0}, \ord_{R_1}, \ord_{R_2}\}$.
 We compute  that  
$\Proj R[It]$ is normal and has precisely one singular point. 

We use that $K := (z^3, xz^2, xyz, y^3, x^3y, x^5)R$  is a monomial  reduction of $I$ 
and check the affine charts associated to  each of the monomial generators of $K$.   
The affine chart
 $A:= R[\frac{I}{y^3}] = R[ \frac{z}{y}, \frac{xz}{y^2}, \frac{x^2}{y}, \frac{x^3}{y^2}]$ is the only affine chart that has a 
 singularity.  
    We compute  that the affine chart  $A$ is a 3-dimensional normal  domain, 
  and  prove below that  the maximal ideal $\m_A :=  (\m, \frac{z}{y}, \frac{xz}{y^2}, \frac{x^2}{y}, \frac{x^3}{y^2})A$ is the unique singular point of $A$.

  Observe that $\ord_{R_0}$ and $\ord_{R_2}$ contain $A$, but $\ord_{R_1}$ does not. The center of $\ord_{R_0}$ on $A$ is the 
  height-one prime ideal $\p_0 := (\m,  \frac{x^2}{y}, \frac{x^3}{y^2})A$, and the center of $\ord_{R_2}$ on $A$ is the 
  height-one prime ideal $\p_2 := (\m, \frac{z}{y}, \frac{x^2}{y})A$.

We  show that the divisor class group of $A$ is an infinite cyclic group. 
Since $A[\frac{1}{y}] = R[\frac{1}{y}]$,  a theorem of  Nagata   (\cite[Theorem 6.3, p. 17]{PS}) 
implies that  the divisor class group $\Cl(A)$   of $A$  is generated by the minimal primes of $yA$.
The minimal primes of $yA$ and $y^3A = IA$ are equal. Therefore, we see that $\Min yA  = \{ \p_0, \, \p_2 \}$. 
Since $\ord_{R_0} (y) = 1$ and $\ord_{R_2} (y) = 3$ and $yA$ is an unmixed height $1$ ideal, we have $yA = \p_0 \cap \p_2^{(3)}$. 
The divisor class group $\Cl(A)$ is generated by $[\p_0], [ \p_2]$, where $[ \phantom{p} ]$ represent the class of a height $1$ prime ideal in $\Cl(A)$. 
The equality $yA = \p_0 \cap \p_2^{(3)}$ gives a relation $[\p_0] = -3 [ \p_2]$, and in fact this is the only relation since $\Cl(A)$ is not torsion by Theorem \ref{3.10}(2).
Therefore, we have $\Cl(A) = \langle [ \p_2 ] \rangle$.

To prove  that the singular locus of $A$ is $\m_A$, 
 let $\q$ be a prime ideal of $A$  and consider the  localization $A_{\q}$.
 If $\q$ does not contain  both  $\p_0$ and $\p_2$,  then  by Nagata's Theorem and the relation $[\p_0] = -3 [ \p_2]$,
 the ring 
  $A_{\q}$ has torsion divisor class group.  Theorem \ref{3.10}(2)  then   implies that $A_{\q}$ is regular. 
  Assume that  $\q$ contains both   $\p_0$ and $\p_2$.    Notice that $(\p_0, \p_2, \frac{xz}{y^2})A = \m_A$. 
  Hence   if $\q \ne \m_A$, then    $\frac{xz}{y^2}  \notin \q$. 
In $A [ \frac{y^2}{xz}]$, we have $\frac{y}{x} =   \frac{z}{y} \cdot\frac{y^2}{xz}$ and $\frac{z}{x} = \frac{y}{x} \cdot \frac{z}{y}$. 
This implies that $\m A [ \frac{y^2}{xz}]    = x A [ \frac{y^2}{xz}] $ is principal. 
In particular, we have $\m A_{\q}$ is principal. Therefore, $A_{\q}$ is on the regular model  $X_\Gamma$ and  hence is regular. 
\end{example}

Assume now that $\dim R  = d \ge 4$ and denote $\m = (x, y, z_1, \ldots, z_{d-2})$.  
The structure of the special $*$-simple  ideal $P_2$ 
is similar  to the 3-dimensional case,   but with more generators as we increase $d$.  
The minimal number of generators of $P_2$ is the same as  that  for $\m^2$.  The  difference  
between   $\m^2$ and   $P_2$ is   that
 $x^2$  is replaced by $x^3$ and
$xy$ by $x^2y$.  Thus if $\dim R = d$, then   
$$
P_2 ~=~(x^3, x^2y, x (z_1, \ldots, z_{d-2}), (y, z_1, \ldots, z_{d-2})^2)R
$$
is minimally generated by $\binom{d+1}{2}$ elements.   We have $\Rees P_2 = \{\ord_{R_2},  \ord_R \}$.   

As in the case where $\dim R = 3$,   the affine chart 
$A := R[\frac{P_2}{y^2}]$ of $\Proj R[P_2t]$ contains precisely one prime ideal for which the 
localization of $A$ is not regular.
We have
$$
A ~=  ~R[\frac{P_2}{y^2}] ~ = ~ R[ \frac{x^3}{y^2}, ~  \frac{x^2}{y}, ~ \frac{x z_1}{y^2}, ~ \ldots, ~  \frac{xz_{d-2}}{y^2}, ~ \frac{z_1}{y}, ~ \ldots, ~ \frac{z_{d-2}}{y} ]
$$
and $\m_A := (x, ~ y,  ~ \frac{x^3}{y^2}, ~  \frac{x^2}{y}, ~ \frac{x z_1}{y^2}, ~  \ldots, ~ \frac{x z_{d-2}}{y^2}, ~ \frac{z_1}{y}, ~ \ldots, ~ \frac{z_{d-2}}{y})A$.  
Notice that $P_1 A = y A$, thus $A$ is also an affine chart on $\Proj R [P_1 P_2 t]$.
We have  $\ord_{R_0}$ and $\ord_{R_2}$ contain $A$,  while $\ord_{R_1}$ does not.  
The center of $\ord_{R_0}$ on $A$ is the 
  height-one prime ideal $\p_0 := (x, ~y,~  \frac{x^2}{y}, \frac{x^3}{y^2})A$, and the center of $\ord_{R_2}$ on $A$ is the 
  height-one prime ideal $\p_2 := (x,~y,~ \frac{x^2}{y},  ~\frac{z_1}{y}, ~\ldots, ~\frac{z_{d-2}}{y})A$.
  
  The proof given above for the case where $\dim R = 3$ also applies here to show that the divisor class group
  $\Cl(A)$ is the infinite cyclic group generated by $ [ \p_2 ] $.
To prove that $\m_A$ is the unique prime ideal of $A$ at which the localization is not regular, let $\q \in \Spec A$ be such that $A_{\q}$ is not regular.
Since $A_{\q}$ must have nontorsion divisor class group by Theorem~\ref{3.10}, $\q$ contains $\p_0 + \p_2$.
The remaining generators of $\m_A$ are of the form $\frac{x u}{y^2}$, where $u$ varies among the variables $z_1, \ldots, z_{d-2}$.
If $\q \ne \m_A$, then by symmetry of the variables $z_1, \ldots, z_{d-2}$, we may assume $\frac{x z_1}{y^2} \notin \q$.
But in $A [ \frac{y^2}{xz_1}]$, a simple computation as above shows that $\m$ extends to a principal ideal.
Therefore, $A_{\q}$ is on the regular model  $X_\Gamma$ and  hence is regular.  
Thus the maximal ideal $\m_A$ is the unique singular point of $A$.

\section{Singularities on the blowup of finitely supported ideals}

We are interested in  algebraic properties of  the singularities of local rings 
on the normalized blowup of  finitely supported ideals.  
Huneke and Sally in \cite{HS} examine the structure of  2-dimensional normal local 
rings   $S$   that birationally dominate a 2-dimensional regular local ring.  Using algebraic 
techniques,   Huneke and Sally  recover much information that was known from work of Lipman 
and Artin  about the structure of  $S$ 
 such as that $S$ has a rational singularity and minimal multiplicity. 
For example,   they show  that $S$ is Gorenstein if and 
 only if $S$  has multiplicity  at most  $2$  \cite[Cor.~1.6]{HS}. 
 
  Example~\ref{7.10}  is a further discussion of Example~\ref{exam:twoBranches} regarding 
  the blowup of a finitely supported ideal.

\begin{example}  \label{7.10}
Let $(R, \m)$ be a regular local ring with $\m = (x, y, z)$ and let $J = (xz, y^2, z^3, y z^2, x^2 y, x^3)$.
Consider the affine chart of $\Proj R [J t]$ obtained by homogeneous localization at the element $xzt$.
This gives the ring $A = R [\frac{y^2}{x z}, \frac{z^2}{x}, \frac{y z}{x}, \frac{x y}{z}, \frac{x^2}{z}]$.
We observe below that 
$\ord_{R_2}$ and $\ord_{R_4}$ are centered on height $1$ primes of $A$, and $\ord_{R_0}$, $\ord_{R_1}$, and $\ord_{R_3}$ are centered on height $2$ primes of $A$. 
All of these prime ideals  are  contained in the maximal ideal $\m_A$, where
$$
\m_A ~:= ~  (x,~ y,~ z, ~\frac{y^2}{x z},~ \frac{z^2}{x},~ \frac{y z}{x},~ \frac{x y}{z},~ \frac{x^2}{z})A.
$$
Thus $\m_A$ is 
 generated by $\m$ and the five listed ring  generators of $A$ over $R$. 
The powers of $\m_A$ do not define a valuation, since $y \in \m_A \setminus \m_A^2$ and $y^2 \in \m_A^3$ by the relation $x z (\frac{y^2}{x z}) = y^2$. \\

Using the chart  
$$
\begin{array}{c|c|c|c}
	& x 	& y 	& z \\ \hline
\ord_{R_0} & 1	& 1 	& 1 \\ \hline
\ord_{R_1} & 1 & 2 	& 2 \\ \hline
\ord_{R_2} & 2	& 3  & 4 \\ \hline
\ord_{R_3} & 2 	& 2	& 1 \\ \hline
\ord_{R_4} & 4	& 3	& 2,
\end{array}
$$
we compute the centers $Q_i$ of $\ord_{R_i}$ on $A$ for $i \in \{ 0, \dots, 4 \}$. They are
\begin{align*}
Q_0 ~&=  ~ (x,~y,~z, ~\frac{z^2}{x},~ \frac{y z}{x},~ \frac{x y}{z},~ \frac{x^2}{z})A,  \\
Q_1 ~&= ~(x,~y,~z, ~\frac{y^2}{x z},~ \frac{z^2}{x}, ~\frac{y z}{x},~ \frac{x y}{z})A, \\
Q_2 ~&= ~ (x,~y,~z, ~\frac{z^2}{x},~ \frac{y z}{x},~ \frac{x y}{z})A, \\
Q_3~ &= ~ (x,~y,~z, ~\frac{y^2}{x z},~ \frac{y z}{x},~ \frac{x y}{z},~ \frac{x^2}{z})A, \hbox{ and } \\
Q_4 ~ &= ~(x,~y,~z, ~\frac{y z}{x}, ~\frac{x y}{z},~ \frac{x^2}{z})A. 
\end{align*}
Since $\Rees J = \{ \ord_{R_2}, \ord_{R_4} \}$,  the prime ideals $Q_2$ and $Q_4$ are of height~$1$ and  the prime ideals  $Q_0, Q_1, Q_3$ are of height $2$.   
The ideal  $L   := (\frac{z^2}{x}, \frac{x^2}{z}, \frac{y^2}{xz}) A$   is  a reduction of $\m_A$. 
Direct computation shows that  $L \m_A = \m_A^2$.
 Therefore the reduction number of  $\m_A$ with respect to $L$ is $1$,  
 and the local ring  $A_{\m_A}$ has minimal multiplicity  
 with Hilbert-Samuel multiplicity  $e(A_{\m_A}) = 6$.

  We have $A = R[\frac{1}{xz}] \cap A_{Q_2} \cap A_{Q_4}$.   The divisor class group  $\Cl(A)$ is generated by the 
 classes $[Q_2]$ and $[Q_4]$.  Notice that $\frac{z^2}{x}A = Q_2^{(6)}$   and  $\frac{x^2}{z}A = Q_4^{(6)}$.  
 Also we have
 $$
 xA ~= ~ Q_2^{(2)} ~\cap ~ Q_4^{(4)}   \quad \text{ and } \quad zA ~=~  Q_2^{(4)} ~\cap ~ Q_4^{(2)}.  
$$
The localization $A[(\frac{z^2}{x})^{-1}]  =    R[\frac{x}{z^2}]  =  R[\frac{1}{xz}] \cap A_{Q_4}$  is contained in $\ord_{R_3}$ and hence is not a UFD
by Corollary~\ref{2.15}.  Since $zA[\frac{x}{z^2}]   = Q_4^{(2)}A[\frac{x}{z^2}]$,  the divisor class group $\Cl(A[\frac{x}{z^2}])$
is a cyclic group of order 2,  and  the divisor class group  $\Cl(A)$ is the direct sum of a cyclic group of 
order 6 with a cyclic group of order 2.
\end{example}

With $A$  and $\m_A$  as in Example~\ref{7.10},  we noted above that the powers of the maximal 
ideal of the local domain $S := A_{\m_A}$ do not define a valuation.  The ring $S$ is on the blowup 
$\Proj R[Jt]$ of a finitely supported ideal $J$ of a regular local ring.  
It seems natural to ask:

\begin{question}  \label{7.20}  Let $I$ be a finitely supported ideal of a regular local ring $R$ and 
let $S \in \Proj \overline{R[It]}$.
Under what conditions do the powers of the maximal ideal of $S$ define a valuation?
\end{question}

 Let $R$ be a regular local ring with $\dim R \ge 2$,  and let 
$\Gamma$  be a finite set of infinitely near points to $R$ that satisfies the 3 
conditions of Remark~\ref{Gamma}  and thus is the set of base points of a finitely 
supported ideal of $R$.   We ask:

\begin{question}
Among the finitely supported ideals $I$ of $R$ with $\mathcal{BP} (I) = \Gamma$, how many distinct projective models $\Proj \overline{R [I t]}$ exist?
If $\Gamma$ has $1$ terminal point and $n$ points which are not terminal points, are there precisely $2^n$ distinct such models?
\end{question}


\begin{thebibliography}{GGP0}

\bibitem   {A}{S.S. Abhyankar, On the valuations centered in a local domain,
Amer. J. Math. {\bf 78} (1956) 321-348}.

\bibitem  {C1}{S. Cutkosky,   Complete ideals in algebra and geometry,  
Commutative algebra: syzygies, multiplicities and birational algebra, 27-39 in 
 Contemp. Math. 159,
Amer. Math. Soc. , Providence, RI, 1994.}

\bibitem   {D1}{C. D'Cruz, Integral closedness of $MI$ and the formula of Hoskin and Deligne for finitely supported complete ideals, 
J. of Alg., {\bf 304} (2006), 613--632.}


\bibitem    {HJLS}{W. Heinzer, B. Johnston, D. Lantz, and K. Shah, Coefficient ideal in and blowups of	
a commutative Noetherian domain, J. of Alg., {\bf 162} (1993), 355--391.}

\bibitem  {HK1}{W. Heinzer and M.-K. Kim,  Rees valuations of complete ideals 
in a regular local ring,   Comm. in Algebra,    {\bf 43}  (2015), 3249--3274.}


\bibitem   {HKT}   {W. Heinzer, M.-K. Kim, and M. Toeniskoetter, Finitely supported {$*$}-simple complete ideals in a
              regular local ring, J. of Alg., {\bf 40 } (2014), 76--106.}
             
\bibitem {HR} {W. Heinzer and M. Roitman,  Well-centered overrings of an integral domain,  J. of Alg., {\bf 272}
(2004),  435-455.}


\bibitem {HS}{C. Huneke and J. Sally, Birational Extensions in Dimension Two 
and Integrally Closed Ideals, J. Algebra, {\bf 115} (1988), 481-500.}




\bibitem   {L1969}  {J. Lipman,    Rational singularities, with applications to algebraic surfaces and unique
factorization, Publ. Math. Inst. Haute Etudes Sci., {\bf 36} (1969), 195-279. }




\bibitem    {L}{J. Lipman, 
On complete ideals in regular local rings,
Algebraic Geometry and Commutative Algebra in Honor of Masayoshi Nagata, (1986), 203-231.}

\bibitem   {M}{H. Matsumura, {\em Commutative Ring Theory},
Cambridge Univ. Press, Cambridge, 1986.}


\bibitem   {N}{M. Nagata, {\em Local Rings}, Interscience, New York, 1962. }




\bibitem  {Sally}{ J. Sally, Fibers Over Closed Points of Birational Morphisms of Nonsingular Varieties, 
Amer. J.  Math., {\bf 104} (1982), 545-552.}





\bibitem  {S}{D. Shannon, 
Monoidal transforms of regular local rings,
Amer. J. Math. {\bf 95} (1973), 294-320.}


\bibitem   {PS}{P. Samuel, {\em Lectures on Unique Factorization Domains},
Tata Institute of Fundamental Research, Bombay, 1964.}


\bibitem {SH}{I. Swanson and C. Huneke, {\it Integral Closure of Ideals,
Rings, and Modules},  London Math. Soc. Lecture Note Series 336, Cambridge
Univ.  Press, Cambridge, 2006.}



\bibitem  {ZS2}
O. Zariski and P. Samuel,{\em  Commutative Algebra, Vol. 2},
D. Van Nostrand, New York, 1960.


\end{thebibliography}
\end{document}